\documentclass{amsart}

\usepackage[utf8]{inputenc}
\usepackage{verbatim}
\usepackage{graphicx}
\usepackage{graphicx,caption2,psfrag,float,color}
\usepackage{amssymb}
\usepackage{amscd}
\usepackage{amsmath}

\newtheorem{theorem}{Theorem}[section]

\newtheorem{lemma}[theorem]{Lemma}

\numberwithin{equation}{subsection}
\newtheorem{definition}[theorem]{Definition}

\newcommand{\R}{\mathbb R}
\newcommand{\F}{\mathbb F}

\title{Cotangent sums related to the Riemann Hypothesis for various shifts of the argument}
\author{Helmut Maier and Michael Th. Rassias}
\date{\today}
\address{Department of Mathematics, University of Ulm, Helmholtzstrasse 18, 89081 Ulm, Germany.}
\email{helmut.maier@uni-ulm.de}
\address{Institute of Mathematics, University of Zurich, CH-8057, Zurich, Switzerland 
\&
Moscow Institute of Physics and Technology,
141700 Dolgoprudny,
Institutskiy per, d. 9,
Russia 
\& Institute for Advanced Study, Program in Interdisciplinary Studies,
1 Einstein Dr, Princeton, NJ 08540, USA.}
\email{michail.rassias@math.uzh.ch}
\thanks{}

\begin{document}

 \maketitle
 
\begin{abstract} One of the approaches to the Riemann Hypothesis is the Nyman-Beurling criterion. Cotangent sums play a significant role. Here we investigate the values of these cotangent sums for various shifts of the argument. 
28A25,\,\,46E15,\,\,60E10.%
\newline

\end{abstract}

\section{Introduction}
In several papers  (\cite{mr}, \cite{mr2}, \cite{mr3}, \cite{mr4}, \cite{max}) the authors have investigated the distribution of the cotangent sums 
%
$$c_0\left(\frac{r}{b}\right):=-\sum_{m=1}^{b-1}\frac{m}{b}\cot\left(\frac{\pi mr}{b}\right)\:,$$
where $r$, $b\in\mathbb{N}$, $b\geq 2$, $1\leq r\leq b$ and $(r,b)=1$.\\
They could establish a link with the function $g(\alpha)$, defined by
\begin{definition}\label{def11}
$$g(\alpha):=\sum_{l=1}^{+\infty}\frac{1-2\{l\alpha\}}{l}\:,\ \alpha\in(0,1)\:,$$
which is convergent for almost all $\alpha$ (see \cite{bre}) .
\end{definition}
\begin{definition}\label{def12}
For $z\in\mathbb{R}$, let 
$$F(z):=\text{meas}\{\alpha\in(0,1)\::\: g(\alpha)\leq z  \},$$
where $``\text{meas}"$ denotes the Lebesgue measure.\\
Let $\mu$ be the uniquely defined positive measure on $\R$ with the following property:\\
For $\alpha<\beta\in\R$ we have:
$$\mu([\alpha,\beta])=F(\beta)-F(\alpha)\:.$$
We set
$$C_0(\mathbb{R}):=\{f\in C(\mathbb{R})\::\: \forall\: \epsilon>0,\: \exists\:\text{a compact set}\ \mathcal{K}\subset\mathbb{R},\:\text{such that}\ |f(x)|<\epsilon,\forall\: x\not\in \mathcal{K}  \}.$$
\end{definition}
The second author in his thesis \cite{rasthesis} (see also \cite{mr}, Theorem 1.2) could establish the following result:\\
Let $A_0,$ $A_1$ be fixed constants, such that $1/2< A_0<A_1<1$. For all $f\in C_0(\mathbb{R})$, we have
\[
\lim_{b\rightarrow+\infty}\frac{1}{(A_1-A_0)\phi(b)}\sum_{\substack{r\::\: (r,b)=1\\ A_0b\leq r\leq A_1b}}f\left( \frac{1}{b}c_0\left( \frac{r}{b}\right) \right)=\int f\:d\mu, \tag{1.1}
\]
where $\phi(\cdot)$ denotes the Euler phi-function.\\
Later S. Bettin \cite{dbcot} could replace the inequality $1/2< A_0<A_1<1$ by $0<A_0<A_1\leq 1$. In \cite{cotprimes} the authors considered the distribution of the values of $c_0$ for rational numbers with primes as numerator and a fixed prime as denominator and could prove a result analogous to (1.1) (Theorem 1.3 of \cite{cotprimes}). Namely, they proved that for all $f\in C_0(\mathbb{R})$, the following holds true:
$$\lim_{\substack{q\rightarrow+\infty\\ q\ \text{prime}}}\frac{\log q}{(A_1-A_0)q}\ \sum_{\substack{p\::\: A_0q\leq p\leq A_1q\\ p\ \text{prime}}}f\left( \frac{1}{q}c_0\left( \frac{p}{q}\right) \right)=\int f\:d\mu\:.$$
The cotangent sum $c_0(r/b)$ has gained importance in the Nyman-Beurling criterion for the Riemann Hypothesis through its relation with the Vasyunin sum, which is defined by:
$$V\left(\frac{r}{b}\right):=\sum_{m=1}^{b-1}\left\{\frac{mr}{b}\right\}\cot\left(\frac{\pi mr}{b}\right)\:,$$
where $\{u\}:=u-\lfloor u \rfloor$, $u\in\mathbb{R}.$\\
We have (see \cite{BETT2}, \cite{BEC}) that 
$$V\left(\frac{r}{b}\right)=-c_0\left(\frac{\bar{r}}{b}\right),$$
where $\bar{r}$ is such that $\bar{r} r\equiv1\:(\bmod\; b)$.\\
One has the following identity
(see \cite{BETT2}, \cite{BEC}):
\begin{align*}
\frac{1}{2\pi(rb)^{1/2}}\int_{-\infty}^{+\infty}\left|\zeta\left(\frac{1}{2}+it\right)\right|^2\left(\frac{r}{b}\right)^{it}\frac{dt}{\frac{1}{4}+t^2}&=\frac{\log 2\pi -\gamma}{2}\left(\frac{1}{r}+\frac{1}{b} \right)\tag{1.2}\\
&+\frac{b-r}{2rb}\log\frac{r}{b}-\frac{\pi}{2rb}\left(V\left(\frac{r}{b}\right)+V\left(\frac{b}{r}\right)\right).
\end{align*}
Terms of the form (1.2) appear in the  Nyman-Beurling-Ba\'ez-Duarte-Vasyunin approach to the Riemann Hypothesis (see \cite{bag}, \cite{BETT2}). The Riemann Hypothesis is true if and only if 
$$\lim_{N\rightarrow+\infty}d_N=0,$$
where
$$d_N^{2}:=\inf_{D_N}\frac{1}{2\pi}\int_{-\infty}^{+\infty}\left|1-\zeta\left(\frac{1}{2}+it\right) D_N\left(\frac{1}{2}+it\right)\right|^2\frac{dt}{\frac{1}{4}+t^2}\:,$$
the infimum being taken over all Dirichlet polynomials
$$D_N(s):=\sum_{n=1}^{N}\frac{a_n}{n^s}.$$
In this paper we simultaneously consider the values of $c_0$ for various shifts of the argument. We consider general numerators as well as prime numerators. For simplicity the denominator will be a fixed prime $q$. Our main results are the following:
\begin{theorem}\label{thm11}
Let $A_0,$ $A_1$ be fixed constants, such that $1/2< A_0<A_1<1$. Let $a_1, a_2, \ldots, a_L$ be distinct non-negative integers. Let $f_1, f_2, \ldots, f_L\in C_0(\R)$ $(L\in\mathbb{N})$. Then we have:\\
\noindent (i) 
$$\lim_{\substack{q\rightarrow+\infty\\ q\ \text{prime}}}\frac{1}{\phi(q)}\sum_{{r\::\: A_0q\leq r\leq A_1q}}\frac{1}{(A_1-A_0)^L}\:\prod_{l=1}^L\:f_l\left( c_0\left( \frac{r+a_l}{q}\right) \right)=\prod_{l=1}^L \left( \int f_l\:d\mu\right)\:.$$
\noindent (ii)
$$\lim_{\substack{q\rightarrow+\infty\\ q\ \text{prime}}}\frac{\log q}{q}\sum_{\substack{p\::\: A_0q\leq p\leq A_1q \\ p\:\text{prime}}}\frac{1}{(A_1-A_0)^L}\:\prod_{l=1}^L\:f_l\left( c_0\left( \frac{p+a_l}{q}\right) \right)=\prod_{l=1}^L \left( \int f_l\:d\mu\right)\:.$$
\end{theorem}
The proof of (ii) can be obtained from the proof of (i) by only minor changes. We give only a detailed proof of (i) and sketch the changes needed for the proof of (ii).
\section{Outline of the proof}
Several fundamental ideas already appear in the paper \cite{mr} and in the thesis \cite{rasthesis}. The key to the treatment of the sum $c_0(r/q)$ lies in its relation to the sum 
$$Q\left(\frac{r}{q}\right):=\sum_{m=1}^{q-1}\cot\left(\frac{\pi mr}{q}\right)\left\lfloor \frac{mr}{q}\right\rfloor\:. $$
The second author in his thesis \cite{rasthesis} established the following (see also
Proposition 1.8 of \cite{mr}):
\[
c_0\left(\frac{r}{q}\right)=\frac{1}{r}\:c_0\left(\frac{1}{q}\right)-\frac{1}{r}Q\left(\frac{r}{q}\right)\:.\tag{2.1}
\]
We now have to consider simultaneously the values
$$f_1\left(Q\left(\frac{r+a_1}{q}\right)\right)\:,\ldots, f_L\left(Q\left(\frac{r+a_L}{q}\right)\right)\:.$$
By the Weierstrass approximation theorem this question may be reduced to the study of the joint distribution of the products 
\[
\prod(k_1, \ldots, k_L):= Q\left(\frac{r+a_1}{q}\right)^{k_1}\cdots Q\left(\frac{r+a_L}{q}\right)^{k_L}  \tag{2.2}
\]
for $L$-tuplets $(k_1, \ldots, k_L)$ of non-negative integers.\\
In a similar fashion as in the previous papers, we shall break up the range of summation into subintervals in which 
$$\left\lfloor \frac{(r+a_l)m}{q}\right\rfloor$$
assumes constant values.
\begin{definition}\label{def21}
For $j\in\mathbb{N}$, $l\in\{1,\ldots, L\}$ we set:
$$S_j^{(l)}:=\{ (r+a_l)m\::\: qj\leq (r+a_l)m<q(j+1)\}$$
and write
$$S_j^{(l)}:=\{ qj+s_j^{(l)}, qj+s_j^{(l)}+(r+a_l) ,\ldots, qj+s_j^{(l)}+(r+a_l)d_j^{(l)}  \}\:.$$
We also define $t_j^{(l)}$ by
$$qj+s_j^{(l)}+d_j^{(l)}(r+a_l)+t_j^{(l)}:=q(j+1)\:.$$
\end{definition}
In \cite{mr} and \cite{cotprimes} the map
$$j\longrightarrow s_j$$
and its inverse
$$s\longrightarrow j(s)$$
were very important.\\
They are now replaced by $L$ maps
$$j\longrightarrow s_j^{(l)}\ \ (1\leq l\leq L)$$
and their inverses
$$s\longrightarrow j_l(s)$$
We also have $L$ pairs of congruences
\[
s_j^{(l)}\equiv -qj\:(\bmod\: (r+a_l))\ \ \text{and}\ \ t_j^{(l)}\equiv q(j+1)\:(\bmod\: (r+a_l))\:.\tag{2.2}
\]
Each of the sums $$Q\left( \frac{r+a_l}{q} \right)$$ is dominated by small values of $s_j^{(l)}$ and $t_j^{(l)}$ because of the poles of the function $\cot(\pi x)$ at $x=0$ and $x=1$. We denote this partial sum by $$Q_0\left( \frac{r+a_l}{q} \right)$$ and thus we obtain the decomposition 
$$Q\left( \frac{r+a_l}{q} \right)=Q_0\left( \frac{r+a_l}{q} \right)+Q_1\left( \frac{r+a_l}{q} \right)\:.$$
The function $\cot(\pi x)$ is antisymmetric
$$\cot(\pi(1-x))=-\cot(\pi x)\:.$$
Therefore there will be considerable cancellation in the sums $$Q_1\left( \frac{r+a_l}{q} \right)\:,$$ which thus will be small.\\
By the binomial theorem the products 
$$\prod (k_1,\ldots, k_L)$$
in (2.1) will be linear combinations of products of the form
\[
Q_{\epsilon_1}\left( \frac{r+a_{l_1}}{q} \right)^{h_1}\cdots\: Q_{\epsilon_M}\left( \frac{r+a_{l_M}}{q} \right)^{h_M}  \tag{2.3}
\]
with $\epsilon_g\in\{0,1\}\:.$\\
We shall show that only the products with $\epsilon_g=0$ for $1\leq g\leq M$ will give a substantial contribution.\\
The asymptotic size of these products will be determined by localising the solutions of the congruences (2.2) simultaneously for all $l$. Whereas in \cite{mr} Kloosterman sums could be used for this localisation, here we need results on more general exponential sums in finite fields, due to Bombieri \cite{bombieri}.\\
The contribution of the other products in (2.3) is small, since at least one $\epsilon_g=1$.\\
For the discussion of these factors 
$$Q_{1}\left( \frac{r+a_{l_g}}{q} \right)^{h_g}$$
we can refer to results of \cite{mr}.
\section{Exponential sums in finite fields}
The next result is due to Bombieri \cite{bombieri}. We largely adopt the definitions and results of this paper.
\begin{definition}\label{def31}
Let $p$ be a prime number, $q=p^m$ be a prime power, $k$ the finite field with $q$ elements. Let $X$ be a complete non-singular curve of genus $g$, defined over $k$.\\
Let ${R}(x)\in k(x)$ be a rational function on $X$, satisfying the condition 
\[
R(x)\neq h^p-h\ \ \text{for}\ h\in\bar{k}(x)\:, \tag{3.1}
\]
$\bar{k}$ being the algebraic closure of ${k}$.
\end{definition}
We may identify $R(x)$ with a rational map
$$R\::\: X\rightarrow A^1$$
of $X$ into the affine line $A^1$ defined over $k$.\\
Let $X_m$ be the set of points of $X$ defined over $F_{q^m}$ (the field of $q^m$ elements), thus we may speak about the value of the function $R(x)$ at a point $x$ of the curve $X$. Let
$$S_m(R,X):=\sum_{x\in X_m} e\left( \frac{\sigma R(x)}{q^m}\right)\:,$$
where $\sigma$ is the absolute trace from $F_{q^m}$ to $k$.
\begin{lemma}\label{lem32}
We have
$$|S_m(R,X)|\leq (d_1^2+2d_1d_2-3d_1)(\sqrt{q})^m+d_1^2\:.$$
\end{lemma}
\begin{proof}
This is Theorem 6 of \cite{bombieri}.
\end{proof}
We now have the following Corollary:
\begin{lemma}\label{lem33}
Let $q$ be a prime number. Let $a_1, \ldots , a_L$ be distinct non-negative integers, $n, m_1, \dots, m_L$ integers not all $0$, 
$$R(x)=nx+\frac{m_1}{x+a_1}+\cdots+\frac{m_L}{x+a_L}\:.$$
Then we have:
$$\sum_{\substack{x=1 \\ x\neq -a_l,\ 1\leq l\leq L}}^{q-1} e\left( \frac{R(x)}{q} \right)=O(q^{1/2})\:.$$
The constant implied by the $O$-symbol may depend on $L$.
\end{lemma}
\begin{lemma}\label{lem34}
Let $\F_r$ be the finite field with $r$ elements and $\psi$ be a non-trivial additive character over $\F_r$, $f$ a rational function of the form 
$$f(x)=\frac{P(x)}{Q(x)}\:,$$
$P$ and $Q$ relatively prime monic polynomials,
$$S(f;r,x):=\sum_{p\leq x} \psi(f(p))\:.$$
($p$ denotes the $p$-fold sum of the element $1$ in $\F_r$).\\
Then we have
$$S(f;r,x)\ll r^{3/16+\epsilon}\: x^{25/32}\:.$$
The implied constant depends only on $\epsilon$ and the degrees of $P$ and $Q$.
\end{lemma}
\begin{proof}
This is due to Fouvry and Michel \cite{fouvry_michel}.
\end{proof}
\section{Localizations of the solutions of the congruences}
\begin{lemma}\label{lem41}
Let $q$ be prime. Let $1/2<A_0<A_1<1$ and $r\in\mathbb{N}$. Let $q^*(r;l)$ be defined by $qq^*(r;l)\equiv 1\:(\bmod\: (r+a_l))$. Let $\alpha_1, \ldots \alpha_L\in(0,1)$, $\delta>0$, such that 
$$\alpha_l+\delta<1\ \text{for}\ 1\leq l\leq L\:.$$
Then we have
\begin{align*}
N(\alpha_1, \ldots \alpha_L, \delta)&:=\left|\left\{ r\::\: r\in\mathbb{N},\ A_0q\leq r\leq A_1q,\ \alpha_l\leq \frac{q^*(r;l)}{r}\leq \alpha+\delta\right\}\right|\\
&=\delta^L(A_1-A_0)q(1+o(1))\:,\ q\rightarrow \infty\:.
\end{align*}
\end{lemma}
\begin{proof}
In the sequel we assume $1\leq l\leq L$. We let $(r+a_l)^*$ be determined by 
$$(r+a_l)(r+a_l)^*\equiv 1\:(\bmod\: q)\:.$$
The Diophantine equation
$$qx+(r+a_l)y=1$$
has exactly one solution $(x_{0,l}\ ,\ y_{0,l})$ with 
$$-\left\lfloor \frac{r+a_l}{2}\right\rfloor<x_{0,l}\leq \left\lfloor \frac{r+a_l}{2}\right\rfloor,\ \ -\left\lfloor \frac{q}{2}\right\rfloor <y_{0,l}\leq \frac{q}{2}$$
We have
\[
q^*(r;l)\equiv x_{0,l}\: (\bmod\: (r+a_l))\:, \tag{4.1}
\]
\[
(r+a_l)^*\equiv y_{0,l}\: (\bmod\: q)\:.
\]
Therefore, for $\beta_l\in(-1/2, 1/2)$ and $\delta>0$ with 
$$\beta_l+\delta<\frac{1}{2}\ \ \text{and}\ \ \beta-\delta>-\frac{1}{2}\ \ \text{for}\ \ 1\leq l\leq L$$
we have
\begin{align*}
\tag{4.2}
&\left|\left\{ r\::\: A_0q\leq r\leq A_1q,\ \frac{y_{0,l}}{q}\in[\beta_l, \beta_l+\delta]\right\}\right|\\
\ \ &=\left|\left\{ r\::\: A_0q\leq r\leq A_1q,\ \frac{x_{0,l}}{r}\in[-(\beta_l+\delta), -\beta_l]\right\}\right|+O(1)\\
&=\left|\left\{ r\::\: A_0q\leq r\leq A_1q,\ \frac{q^*(r,l)}{r}\:(\bmod\: 1)\in[-(\beta_l+\delta), -\beta_l]\right\}\right|+O(1)\:,\\
\end{align*}
where 
$$\frac{q^*(r,l)}{r}\:(\bmod\: 1)\in[-(\beta_l+\delta), -\beta_l]$$
stands for 
\begin{equation}
\frac{q^*(r,l)}{r}\in \left\{
\begin{array}{l l}
    [-(\beta_l+\delta), -\beta_l]+1\:, & \quad \text{if}\ \beta_l\geq 0\:,\\
    
    [-(\beta_l+\delta), -\beta_l]\:, & \quad \text{if}\ \beta_l<0\:.\:\\
  \end{array} \right.
 \nonumber
\end{equation}
Let $\Delta>0$, such that $\beta_l+\delta+\Delta\leq \frac{1}{2}$, $0\leq v_l\leq \Delta$.\\
We define the function
\begin{equation}
\tag{4.3} \chi_{1,l}(u,v):= \left\{
\begin{array}{l l}
    1\:, & \quad \text{if}\ u\in[ \beta_l+\Delta-v,\ \beta_l+\delta-\Delta+v) \:,\\
    
    0\:, & \quad \text{otherwise}\:.\:\\
  \end{array} \right.
 \nonumber
\end{equation}
and
\begin{equation}
\tag{4.4} \chi_{2,l}(u,v):= \left\{
\begin{array}{l l}
    1\:, & \quad \text{if}\ u\in[ \beta_l-\Delta-v,\ \beta_l+\delta-\Delta-v) \:,\\
    
    0\:, & \quad \text{otherwise}\:.\:\\
  \end{array} \right.
 \nonumber
\end{equation}
as well as the functions $\lambda_{1,l}$, $\lambda_{2,l}$ by
$$\lambda_{i,l}(u):=\Delta^{-1}\int_0^{\Delta} \chi_{i,l}(u,v)\: dv\ \ \text{for}\ i=1,2\:.$$
Let the function
\begin{equation}
\tag{4.5} \tilde{\chi}_l(r,\beta):= \left\{
\begin{array}{l l}
    1\:, & \quad \text{if}\ \frac{(r+a_l)^*}{q}\in[ \beta_l, \beta_l+\delta] \:,\\
    
    0\:, & \quad \text{otherwise}\:.\:\\
  \end{array} \right.
 \nonumber
\end{equation}
Since $\lambda_{i,l}$ for $i=1,2$ is obtained from $\chi_{i,l}$ by
averaging over $r$ and since $0\leq \chi_{i,l}(u,v)\leq 1$ it follows that $0\leq \lambda_{i,l}(u)\leq 1$ for $i=1,2$.\\
From (4.3) we have
$$\lambda_{1,l}\left( \frac{(r+a_l)^*}{q} \right)=0\:,\ \ \text{if}\ \ \frac{(r+a_l)^*}{q} \not\in [\beta_l, \beta_l+\delta)\:.$$
Similarly, from (4.4) we have
$$\lambda_{2,l}\left( \frac{(r+a_l)^*}{q} \right)=1\:,\ \ \text{if}\ \ \frac{(r+a_l)^*}{q} \not\in [\beta_l, \beta_l+\delta)\:.$$
Thus we obtain 
\[
\lambda_{1,l}\left( \frac{(r+a_l)^*}{q} \right) \leq \tilde{\chi}_l(r,\beta)\leq \lambda_{2,l}\left( \frac{(r+a_l)^*}{q}\right) \:. \tag{4.6}
\]
We have the Fourier expansion
$$\lambda_{i,l}(u)=\sum_{n=-\infty}^\infty a_l(n) e(nu)\:.$$
The Fourier coefficients $a_l(n)$ are computed as follows:\\
For $i=1$:
$$a_l(0)=\Delta^{-1}\int_0^\Delta\left( \int_{\beta_l-\Delta+v}^{\beta_l+\delta+\Delta-v} 1\: du\right)\: dv=\delta+\Delta\:,$$
as well as 
\begin{align*}
&a_l(n)=\Delta^{-1}\int_0^\Delta\left( \int_{\beta_l-\Delta+v}^{\beta_l+\delta+\Delta-v} e(-nu)\: du\right)\: dv\\
&=\Delta^{-1}\int_0^\Delta-\frac{1}{2\pi i n}(e(-n(\beta_l+\delta+\Delta-v))-e(-n(\beta_l-\Delta+v)))\: dv\\
&=-\frac{1}{4\pi^2n^2}\: \Delta^{-1} (e(-n(\beta_l+\delta))-e(-n(\beta_l+\delta+\Delta))-e(-n\beta_l)+e(-n(\beta_l-\Delta)))\:.
\end{align*}
From the above and an analogous computation, for $i=2$ we obtain
$$a_l(0)=\delta+R_{1,l},\ \text{where}\ |R_{1,l}|\leq \Delta$$
and
\begin{equation}
\tag{4.7} a_l(n)= \left\{
\begin{array}{l l}
    O(\Delta)\:, & \quad \text{if}\ |n|\leq \Delta^{-1} \:,\\
    
    O(\Delta^{-1}n^2)\:, & \quad \text{if}\ |n|>\Delta^{-1}\:.\:\\
  \end{array} \right.
 \nonumber
\end{equation}
Let $\Delta_1>0$, such that $A_0-\Delta_1>1/2$, $A_1+\Delta_1<1$ and $0\leq v\leq \Delta_1$.\\
We define the functions
\begin{equation}
\tag{4.8} \chi_3(u,v):= \left\{
\begin{array}{l l}
    1\:, & \quad \text{if}\ u\in[A_0+v-\Delta_1, A_1-v+\Delta_1] \:,\\
    
    0\:, & \quad \text{otherwise}\:.\:\\
  \end{array} \right.
 \nonumber
\end{equation}
and
\begin{equation}
\tag{4.9} \chi_4(u,v):= \left\{
\begin{array}{l l}
    1\:, & \quad \text{if}\ u\in[A_0+\Delta_1-v, A_1+\Delta_1+v] \:,\\
    
    0\:, & \quad \text{otherwise}\:.\:\\
  \end{array} \right.
 \nonumber
\end{equation}
as well as the functions $\lambda_3, \lambda_4$ by
$$\lambda_i(u):=\Delta_1^{-1}\int_0^{\Delta_1} \chi_i(u,v)\: dv\ \ \text{for}\ i=3,4\:$$
Let the function
\begin{equation}
\tag{4.10} \chi^*(r,\beta):= \left\{
\begin{array}{l l}
    1\:, & \quad \text{if}\ A_0\leq \frac{r}{q}\leq A_1 \:,\\
    
    0\:, & \quad \text{otherwise}\:.\:\\
  \end{array} \right.
 \nonumber
\end{equation}
Since $\lambda_i$ for $i=3,4$ is obtained from $\chi_i$ by averaging over $r$ and since 
$$0\leq \chi_i(u,v)\leq 1\ \ \text{for}\ i=3,4$$
we obtain $0\leq \lambda_i(u)\leq 1$ for $i=3,4$.\\
From the above and an analogous computation, for $i=1,2$ we obtain 
$$a_l(0)=\delta+R_{1,l},\ \text{where}\ |R_{1,l}|\leq \Delta$$
and
\begin{equation}
\tag{4.7} a_l(n):= \left\{
\begin{array}{l l}
    O(\Delta)\:, & \quad \text{if}\ |n|\leq \Delta^{-1} \:,\\
    
    O(\Delta^{-1}n^2)\:, & \quad \text{if}\ |n|>\Delta^{-1}\:.\\
  \end{array} \right.
 \nonumber
\end{equation} 
Let $\Delta_1>0$, such that $A_0-\Delta_1>1/2$, $A_1+\Delta_1<1$ and $0\leq v\leq \Delta_1$.\\
We define the functions
\begin{equation}
\tag{4.8} \chi_3(u,v):= \left\{
\begin{array}{l l}
    1\:, & \quad \text{if}\ u\in [A_0+v-\Delta_1, A_1-v+\Delta_1] \:,\\
    
    0\:, & \quad \text{otherwise}\:.\\
  \end{array} \right.
 \nonumber
\end{equation} 
and
\begin{equation}
\tag{4.9} \chi_4(u,v):= \left\{
\begin{array}{l l}
    1\:, & \quad \text{if}\ u\in [A_0+\Delta_1-v, A_1+\Delta_1+v] \:,\\
    
    0\:, & \quad \text{otherwise}\:.\\
  \end{array} \right.
 \nonumber
\end{equation} 
as well as the functions $\lambda_3, \lambda_4$ by
$$\lambda_i(u):=\Delta_1^{-1}\int_0^{\Delta_1} \chi_i(u,v)\: dv\ \text{for}\ i=3,4\:.  $$
Let the function
\begin{equation}
\tag{4.10} \chi^*(r,\beta):= \left\{
\begin{array}{l l}
    1\:, & \quad \text{if}\ A_0\leq \frac{r}{q}\leq A_1 \:,\\
    
    0\:, & \quad \text{otherwise}\:.\\
  \end{array} \right.
 \nonumber
\end{equation} 
Since $\lambda_i$ for $i=3,4$ is obtained from $\chi_i$ by averaging over $v$ and since
$$0\leq \chi_i(u,v)\leq 1\ \ \text{for}\ i=3,4$$
we obtain 
$$0\leq \lambda_i(u)\leq 1\ \ \text{for}\ i=3,4\:.$$ 
From (4.8) we have
$$\lambda_3\left(\frac{r}{q}\right)=0\:,\ \ \text{if}\ \frac{r}{q}\not\in(A_0,A_1)\:.$$
From (4.9) we have
$$\lambda_3\left(\frac{r}{q}\right)=1\:,\ \ \text{if}\ \frac{r}{q}\in(A_0,A_1)\:.$$
Therefore, we obtain
\[
\lambda_3\left(\frac{r}{q}\right)\leq \chi^*(r,\beta)\leq \lambda_4\left(\frac{r}{q}\right)\:.  \tag{4.11}
\]
By an analogous computation as for $\lambda_1, \lambda_2$ we obtain the Fourier expansions
$$\lambda_i(u)=\sum_{n=-\infty}^\infty c(n)e(nu)\:,\ \ \text{for}\ i=3,4,$$
with $c(0)=A_1-A_0+R_2$, where $|R_2|\leq \Delta_1$ and
\begin{equation}
\tag{4.12} c(n):= \left\{
\begin{array}{l l}
    O(1)\:, & \quad \text{if}\ |n|\leq \Delta_1^{-1} \:,\\
    
    O(\Delta_1^{-1}n^{-2})\:, & \quad \text{if}\ |n|>\Delta_1^{-1}\:.\\
  \end{array} \right.
 \nonumber
\end{equation}
From (4.2), (4.5), (4.6), (4.10) and (4.11), setting $\beta=-\alpha$, we get the following
\begin{align*}
\tag{4.13}  &\sum_{1\leq r\leq q-1}\left( \prod_{l=1}^L \lambda_{1,l} \left( \frac{(r+a_l)^*}{q} \right)\right)\lambda_3\left(\frac{r}{q}\right)\\
&\ \ \leq N(\alpha_1,\ldots, \alpha_L, \delta)\leq  \sum_{1\leq r\leq q-1}\left( \prod_{l=1}^L \lambda_{2,l} \left( \frac{(r+a_l)^*}{q} \right)\right)\lambda_4\left(\frac{r}{q}\right)\:.
\end{align*}
We obtain
\begin{align*}
\tag{4.14}  &\sum_{1\leq r\leq q-1}\left( \prod_{l=1}^L \lambda_{1,l} \left( \frac{(r+a_l)^*}{q} \right)\right)\lambda_3\left(\frac{r}{q}\right)\\
&\ \ = \sum_{m_1,\ldots,m_L,n=-\infty}^\infty a(m_1)a(m_2)\cdots a(m_L)c(n)E(n,m_1,\ldots,m_L,q) \:,
\end{align*}
with
$$E(n,m_1,\ldots,m_L,q):=\sum_{1\leq r\leq q-1} e\left( \frac{nr+m_1(r+a_1)^*+\cdots+m_L(r+a_L)^*}{q} \right)\:.$$
For $(n,m_1,\ldots,m_L)\neq (0,0,\ldots,0)$ we estimate $E(n,m_1,\ldots,m_L,q)$ by Lemma \ref{lem33} and obtain
$$E(n,m_1,\ldots,m_L,q)=O(q^{1/2})\:.$$
From (4.14) we get:
\[
\sum_{1\leq r\leq q-1}\left( \prod_{l=1}^L \lambda_{1,l} \left( \frac{(r+a_l)^*}{q} \right)\right)\lambda_3\left(\frac{r}{q}\right)=(\delta+R_1)^L(A_1-A_0+R_2)q+o(q)\:, \tag{4.15}
\]
for $|R_1|\leq \Delta$ and $|R_2|\leq \Delta$. The same computation also gives:
\[
\sum_{1\leq r\leq q-1}\left( \prod_{l=1}^L \lambda_{2,l} \left( \frac{(r+a_l)^*}{q} \right)\right)\lambda_4\left(\frac{r}{q}\right)=(\delta+R_1)^L(A_1-A_0+R_2)q+o(q)\:. \tag{4.16}
\]
Since $\Delta$ and $\Delta_1$ can be chosen to be arbitrarily small, it follows that (4.15) and (4.16) imply Lemma \ref{lem41}.
\end{proof}
\section{Decomposition of the sums $Q$}
We start from the decompositions 
\[
Q\left(\frac{r+a_l}{q}\right)=\sum_{j=1}^{r-1}\:j\:\sum_{h=0}^{d_j}\cot\left(\pi\: \frac{s_j^{(l)}+hr}{q}\right)\:. 
\]
We further decompose $Q\left(\frac{r+a_l}{q}\right)$ as in the following definition:
\begin{definition}\label{def51}
\[
Q\left(\frac{r+a_l}{q}\right):=Q_0\left(\frac{r+a_l}{q}\right)+Q_1\left(\frac{r+a_l}{q}\right)\tag{5.1}
\]
with
\[
Q_0\left(\frac{r+a_l}{q}\right):=\sum_{j=1}^{q-1}{}^*\:j\:\sum_{h=0}^{d_j^{(l)}}\cot\left(\pi\: \frac{s_j^{(l)}+hr}{q}\right)  
\]
where $\sum_{j=1}^{q-1}{}^*$ means that the sum is extended over all values of $j$, for which
$$\left\{\frac{\theta jq}{r+a_j}\right\}\leq q^{-1} 2^{m_1}$$
for either $\theta=1$ or $\theta=-1$,
$$Q_1\left(\frac{r+a_l}{q}\right)=Q\left(\frac{r+a_l}{q}\right)-Q_0\left(\frac{r+a_l}{q}\right)\:,$$
$m_1$ is a fixed positive integer.
\end{definition}
(For the conclusion of the proof we let $m_1\rightarrow \infty$).\\
The size of $Q\left(\frac{r+a_l}{q}\right)$ and also of $c_0\left(\frac{r+a_l}{q}\right)$ is essentially determined by $Q_0$, since $Q_1$ is small, as we shall see in Section 7.
\section{Comparison of $Q_0$ and $g$}
\begin{definition}\label{def61}
Let $A_0q\leq r\leq A_1q$. We set
$$\alpha^{(l)}:=\alpha^{(l)}(r,q)=\frac{q_l^*}{r+a_l}\:,$$
where
$$q_l^*q\equiv 1\:(\bmod\: (r+a_l))\:,$$
$$g(\alpha; m_1):=\sum_{s=1}^{2^{m_1}} \frac{1-2\{s\alpha\}}{s}\:,$$
$$Q(r,q,m_1,l):=\frac{rq}{\pi}\: g(\alpha^{(l)}; m_1)\:.$$
\end{definition}
The next lemma shows, that $Q_0\left(\frac{r+a_l}{q}\right)$ is well approximated by $Q(r,q,m_1,l)$.
\begin{lemma}\label{lem62}
$$Q_0\left(\frac{r+a_l}{q}\right)=Q(r,q,m_1,l)+O(q2^{m_1})\:.$$
\end{lemma}
\begin{proof}
This follows from the result in the thesis of the second author \cite{rasthesis} and in the paper \cite{mr}, step 1 of the proof of Theorem 4.15, if $r$ is replaced by $r+a_l$.
\end{proof}
\section{The estimate of $Q_1(p/q)$}
\begin{lemma}\label{lem71}
We have
$$Q_1\left(\frac{r+a_l}{q}\right)=O(q^22^{-m_1})\:.$$
\end{lemma}
\begin{proof}
From the formula (4.113) of \cite{mr} we have
$$Q_1\left(\frac{r}{q}\right)=O(q^22^{-m_1})\:.$$
Lemma \ref{lem71} follows, if we replace $r$ by $r+a_l$.
\end{proof}
\section{The joint moments of the sums $Q(r,q,m_1,l)$ and $c_0\left(\frac{r+a_l}{q} \right)$}
\begin{lemma}\label{lem81}
We have
$$\lim_{\substack{q\rightarrow \infty\\ q\ \text{prime}}} (A_1-A_0)^{-L}q^{-L}\prod_{l=1}^L Q(r,q,m_1,l)^{k_l}=\prod_{l=1}^L \left( \int_0^1 g(\alpha, m_1)^{k_l}\: d\alpha\right)\:.$$
\end{lemma}
\begin{proof}
We choose a partition $\mathcal{P}$ of the interval $[0,1]$:
\[
0=\alpha_0<\alpha_1<\cdots <\alpha_{n-1} <\alpha_n=1\: \tag{8.1}
\]
and consider upper and lower Riemann sums of the $L$-dimensional Riemann integral
$$I:=\int_0^1\cdots\int_0^1 g(x^{(1)}, m_1)^{k_1}\cdots g(x^{(L)}, m_1)^{k_L}\: dx^{(1)}\cdots dx^{(L)}\:,$$
the upper sum
\begin{align*}
&\mathcal{U}\left(g(x^{(1)}, m_1)^{k_1},\ldots, g(x^{(L)}, m_1)^{k_L}; \mathcal{P}\right)\\
&\ :=\sum_{i_1=0}^{n-1}\cdots \sum_{i_L=0}^{n-1}\ \ \sup_{(\alpha^{(i_1)},\ldots,\alpha^{(i_L)}\in X_{l=1}^L[\alpha_{i_l}, \alpha_{i_l+1}] } \left(  g(\alpha^{(i_1)}, m_1)^{k_1}\cdots g(\alpha^{(i_L)}, m_1)^{k_L} \right)\\
&\ \ \ \times \prod_{l=1}^L (\alpha_{i_l+1}-\alpha_{i_l})
\end{align*}
and the lower sum
\begin{align*}
&\mathcal{L}\left(g(x^{(1)}, m_1)^{k_1},\ldots, g(x^{(L)}, m_1)^{k_L}; \mathcal{P}\right)\\
&\ :=\sum_{i_1=0}^{n-1}\cdots \sum_{i_L=0}^{n-1}\ \ \inf_{(\alpha^{(i_1)},\ldots,\alpha^{(i_L)}\in X_{l=1}^L[\alpha_{i_l}, \alpha_{i_l+1}] } \left(  g(\alpha^{(i_1)}, m_1)^{k_1}\cdots g(\alpha^{(i_L)}, m_1)^{k_L} \right)\\
&\ \ \ \times \prod_{l=1}^L (a_{i_l+1}-\alpha_{i_l})
\end{align*}
The function $g(x, m_1)$ is piecewise linear. Therefore the integral $I$ exists. It is well known from the Theory of the Riemann-integral, that for given $\epsilon>0$ there is a partition $\mathcal{P}_\epsilon$ of the form (8.1), such that
\[
\mathcal{L}\leq \prod_{l=1}^L \left( \int_0^1 g(x, m_1)^{k_l}\: dx\right)\leq \mathcal{U}\leq \mathcal{L}+\epsilon\:.  \tag{8.2}
\]
We now let
$$N_{i_1,\ldots, i_L}:=\left|\left\{ r\::\: A_0q\leq r\leq A_1q\:,\ \frac{q^*(r+a_l)}{r}\in [\alpha_{i_l}, \alpha_{i_l+1})\ \text{for}\ 1\leq l\leq L\right\}\right|\:.$$
By Lemma \ref{lem41} we have
$$N_{i_1,\ldots, i_L}=\delta^L(A_1-A_0)q(1+o(1))\:.$$
From the asymptotics
$$\cot\left(\frac{\pi (qj+s_j^{(l)})}{q}\right)=\frac{q}{\pi}\:\frac{1}{s_j^{(l)}}(1+o(1))$$
and
$$\cot\left(\frac{\pi (qj+d_j^{(l)}(r+a_l)+t_j^{(l)}}{q}\right)=-\frac{q}{\pi}\:\frac{1}{t_j^{(l)}}(1+o(1))$$
we obtain (using the notation 2.3) the following
\begin{align*}
&\sum_{1\leq r\leq q-1}\ \prod_{l=1}^L Q(r,q,m_1, l)^{k_l}\\
&=\sum_{i_1=0}^{n-1}\cdots \sum_{i_L=0}^{n-1}\ \prod_{l=1}^L \left( \sum_{j\::\: \left\{ \frac{\theta j q}{r+a_l}\right\}\in[\alpha_{i_l}, \alpha_{i_l+1}]\ \text{for}\ \theta\in \{1,-1\}} \cot\left(  \frac{\pi u_j^{(l)}}{q}\right)^{k_l} \right)\\
&\ \ \ \ \ \text{(where we set $u_j^{(l)}=s_j^{(l)}$ if $\theta=1$, and $u_j^{(l)}=-t_j^{(l)}$ if $\theta=-1$)}\\
&=\left( \prod_{l=1}^L\left( \int_0^1 g(\alpha, m_1)^{k_l}\: d\alpha\right)\right)\: q^L(1+o(1))\:,
\end{align*}
which proves Lemma \ref{lem81}.
\end{proof}
\begin{lemma}\label{lem82}
We have
$$\lim_{\substack{q\rightarrow \infty\\ q\ \text{prime}}} (A_1-A_0)^{-L} q^{-L} \prod_{l=1}^L Q\left(\frac{r+a_l}{q}\right)^{k_l}=\prod_{l=1}^L\left(\int_0^1 g(\alpha)^{k_l}\:d\alpha\right)\:.$$
\end{lemma}
\begin{proof}
This follows from Lemma \ref{lem81} by the use of the decomposition (5.1):
\[
Q\left(\frac{r+a_l}{q}\right):=Q_0\left(\frac{r+a_l}{q}\right)+Q_1\left(\frac{r+a_l}{q}\right)\:,
\]
from Lemma \ref{lem62}:
$$Q_0\left(\frac{r+a_l}{q}\right)=Q(r,q,m_1,l)+O(q2^{m_1})\:,$$
Lemma \ref{lem71}, and by the use of the binomial theorem.
\end{proof}
\begin{lemma}\label{lem83}
We have
\begin{align*}
&\lim_{\substack{q\rightarrow \infty\\ q\ \text{prime}}}\frac{1}{\phi(q)}\sum_{r\::\: A_0q\leq r\leq A_1 q} (A_1-A_0)^{-(k_1+\cdots+k_L)}\ \prod_{l=1}^L c_0\left(\frac{r+a_l}{q}\right)^{k_l}\\
&\ \ \ =\prod_{l=1}^L\left(\int_0^1 g(x)^{k_l}\:dx\right)\:.
\end{align*}
\end{lemma}
\begin{proof}
This follows from Lemma \ref{lem82} and formula (2.1) by the use of the binomial theorem.
\end{proof}
\section{Conclusion of the proof and concluding remarks}
By the Weierstrass approximation theorem, each of the functions $f_l$ in Theorem \ref{thm11} can be approximated arbitrarily closely by a polynomial 
$$p_l(x)=\sum_{j=0}^{C_l} e_{j,l}\: x^j\:,$$
where $C_l\in\mathbb{N}_0$.
The product
$$\prod_{l=1}^L f_l\left( c_0\left(\frac{r+a_l}{q}\right)\right)$$
then becomes the sum of products as considered in Lemma \ref{lem83}.\\ 
The Theorem follows from the fact that $g$ has a continuous distribution function (Theorem 5.2 of \cite{mr}).\\
The proof of Theorem \ref{thm11} (ii) can be obtained from the proof of (i) by applying Lemma \ref{lem34} instead of Lemma \ref{lem33} and making the obvious changes otherwise.

\vspace{10mm}
\noindent\textbf{Acknowledgments.}\\
M. Th. Rassias: I would like to express my gratitude to the  John S. Latsis Public Benefit Foundation for their financial support provided under the auspices of my current ``Latsis Foundation Senior Fellowship'' position.


\vspace{10mm}
\begin{thebibliography}{99}%


\bibitem{bag} B. Bagchi, \textit{On Nyman, Beurling and Baez-Duarte's Hilbert space reformulation of the Riemann hypothesis}, Proc. Indian Acad. Sci. Math. 116(2)(2006), 137--146. 



\bibitem{BETT2} S. Bettin, \textit{A generalization of Rademacher's reciprocity law}, Acta Arithmetica, 159(4)(2013), 363--374.
\bibitem{dbcot} S. Bettin, \textit{On the distribution of a cotangent sum}, IMRN, doi: 10.1093/imrn/rnv036, 2015.
\bibitem{BEC} S. Bettin and B. Conrey, \textit{Period functions and cotangent sums}, Algebra \& Number Theory 7(1)(2013), 215--242.

\bibitem{bombieri} E. Bombieri, \textit{On Exponential Sums in Finite Fields}, American Journal of Mathematics, 88(1)(1966), 71--105

\bibitem{bre} R. de la Bret\`eche and G. Tenenbaum, \textit{S\'eries trigonom\'etriques \`a coefficients arithm\'etiques}, J.  Anal. Math., 92(2004), 1--79.



\bibitem{fouvry_michel} E. Fouvry and Ph. Michel, Sur certaines sommes
d'exponentielles sur les nombres premiers, Ann. Sci. \'Ecope Norm. Sup. (4), 31(1998), 93-130.



\bibitem{mr1} H. Maier and M. Th. Rassias, \textit{The rate of growth of moments of certain cotangent sums}, Aequationes Mathematicae, 2015, DOI 10.1007/s00010-015-0361-3.

\bibitem{mr2}  H. Maier and M. Th. Rassias, \textit{The order of magnitude for moments for certain cotangent sums}, Journal of Mathematical Analysis and Applications, 429(1)(2015), 576--590.


\bibitem{mr} H. Maier and M. Th. Rassias, \textit{Generalizations of a cotangent sum associated to the Estermann zeta function}, Communications in Contemporary Mathematics, 18(1)(2016), doi: 10.1142/S0219199715500789.

\bibitem{mr2} H. Maier and M. Th. Rassias, \textit{The order of magnitude for moments for certain cotangent sums}, Journal of Mathematical Analysis and Applications, 429(1)(2015), 576--590.

\bibitem{mr3} H. Maier and M. Th. Rassias, \textit{The rate of growth of moments of certain cotangent sums}, Aequationes Mathematicae, 2015, 90(3)(2016), 581--595.

\bibitem{mr4} H. Maier and M. Th. Rassias, \textit{Asymptotics for moments of certain cotangent sums}, Houston Journal of Mathematics (to appear).

\bibitem{max} H. Maier and M. Th. Rassias, \textit{The maximum of cotangent sums related to Estermann's zeta function in rational numbers in short intervals}, Applicable Analysis and Discrete Mathematics, 11(2017), 166--176.

\bibitem{cotprimes} H. Maier and M. Th. Rassias, \textit{Distribution of a cotangent sum related to the Nyman-Beurling criterion for the Riemann Hypothesis}, preprint.



\bibitem{rasthesis} M. Th. Rassias, \textit{Analytic investigation of cotangent sums related to the Riemann zeta function}, Doctoral Dissertation, ETH-Z\"urich, Switzerland, 2014.

\bibitem{ras} M. Th. Rassias, \textit{On a cotangent sum related to zeros of the Estermann zeta function}, Applied Mathematics and Computation, 240(2014), 161--167.




Character Sums, UNSW Sidney, 2016.


\end{thebibliography}
\end{document}